\documentclass[a4paper, 12pt,reqno]{amsart}
\usepackage{amsmath,amssymb,amsthm,enumerate}
\allowdisplaybreaks
\theoremstyle{plain}
\newtheorem{theorem}{Theorem}
\newtheorem*{theorem*}{Theorem}
\newtheorem{lemma}{Lemma}
\newtheorem*{lemma*}{Lemma}
\theoremstyle{definition}
\newtheorem{definition}{Definition}
\newtheorem*{definition*}{Definition}
\theoremstyle{remark}
\newtheorem{remark}{Remark}
\newtheorem*{remark*}{Remark}
\theoremstyle{example}
\newtheorem{example}{Example}
\newtheorem*{example*}{Example}
\theoremstyle{problem}
\newtheorem{problem}{Problem}
\newtheorem*{problem*}{Problem}
\theoremstyle{hypothesis}
\newtheorem{conjecture}{Conjecture}
\newtheorem*{conjecture*}{Conjecture}
\theoremstyle{corollary}
\newtheorem{corollary}{Corollary}
\newtheorem*{corollary*}{Corollary}

\begin{document}
\title[On one type of modelling certain numeral systems]{On one approach to modeling numeral systems (On one type of modelling certain numeral systems)}
%
\author{Symon Serbenyuk}
\address{ 45 Shchukina St.\\ 
Vinnytsia\\ 
21012 \\
  Ukraine}
\email{simon6@ukr.net}


\begin{abstract}

In this article,  the operator approach to modelling numeral systems  is introduced. This approach can be useful for coding information  and providing computer protection. Certain examples of such numeral systems are considered. In addition, the pseudo-binary representation is investigated. A description of further investigations of the author of this article is given.

\end{abstract}
\maketitle


\section{Introduction}

As a straightforward consequence of the dynamic development in information technology and computing the necessity of safe encoding and information protection has arisen. Modeling  generalizations of well-known numeral systems is an important tool in this case.  For non-cracking the coding mechanism, new numeral systems should be difficult and should be systematically improved by a certain generalization or constructing a chain of sequential generalizations. In addition, various systems of real number encodings are effective tools for modeling and studying ``pathological" mathematical objects (for example,  the notion of  ``pathology" in mathematics is described in \cite{35}). Such mathematical objects (e.g., fractal sets, singular functions, non-differentiable, or nowhere monotonic functions, etc.) have the applied importance and the interdisciplinary character.  Examples of using or applying pathological mathematical objects in one or several fields of science can be easily founded in a number of literatures, even on Wikipedia (for example, see \cite{1, 5, 76, 11,   70, 18,  S21, 36, 37, 38}, etc.). As a drawback, it is possible that the applied importance of such mathematical objects can be justified only with time in one or another field of science. 

One can note that now in science there exists the tendency of modeling and studying numeral systems defined in terms of alternating expansions of real numbers under the condition that numeral systems defined in terms of corresponding positive expansions of real numbers are well-known.  Let us consider several examples:
\begin{itemize}

\item The notion of the $s$-adic numeral system with a fractional base   (more known   as the $\beta$- expansion) was introduced by A. R\'enyi in \cite{Renyi1957} in 1957. However, the nega-$s$-adic numeral system with a fractional base   (or $(-\beta)$-expansion) was introduced in \cite{IS2009} in 2009. Later, the last-mentioned numeral system was investigated by a number of researchers from Western Europe,  USA, and East Asia:  P. Ambro\v{z}, D. Dombek, Ch. Frougny,  A. Ch. Lai,  L. Liao, P. Loreti, Z. Mas\'akov\'a, K. Moore, E. Pelantov\'a, W. Steiner,  and other scientists. 

\item  In 1883, in the paper \cite{15}, the German mathematician J. L\"uroth introduced an expansion of a real number in the form of a special series (now these series are called  L\"uroth series). However, in 1990, S. Kalpazidou, A. Knopfmacher, J. Knopfmacher introduced alternating L\"uroth series in the paper \cite{10}. 
\end{itemize}

Such examples also exist for other expansions, e.g., for the case of positive and alternating Engel series~(for example, see \cite{Engel1, Engel 3} and references in \cite{Engel2, Engel4}).

The main purpose of this paper is to intoduce and to investigate  some  new numeral systems, geometry of which  is a generalization of geometries of  certain  numeral systems with positive or negative base. The main attention is given to certain two examples of numeral systems which are related to the binary and ternaty numeral systems by a certain map. It is proven that any number from $[0,1]$ can be be represented in terms of  these numeral systems, properties of cylinders defined in terms of  the representations and also some properties of the considered maps are  investigated.

The paper is organized as follows. In Section~\ref{section 2}, the fundamental notations and notions which are used throughout the paper, are described. The main definition and explanations of the considered approach for modeling numeral systems, are given in Section~\ref{section 2a}.
Section~\ref{section 3} is devoted to one example of the pseudo-ternary numeral system. Section~\ref{section 3a} includes investigations of the  the pseudo-binary numeral systems. Finally, Section~\ref{section 4} is devoted to  further investigations of the author of this paper.

\section{Preliminaries}
\label{section 2}
\begin{definition}{\cite{39}.}
``\emph{A numeral system}  is a writing system for expressing numbers; that is, a mathematical notation for representing numbers of a given set, using digits or other symbols in a consistent manner." 
\end{definition}

Let us denote by  $\mathbb N$  the set of all positive integers,  by $|\cdot |$ the Lebesgue measure of a set, and by $d(\cdot)$ the diameter of a set (i.e.,  $d(\cdot)=\sup(\cdot)-\inf(\cdot)$). 

Let $s>1$ be a fixed positive integer. It is well known that any number $x\in [0,1]$ can be represented by the following form
\begin{equation}
\label{eq: s-adic-1}
\sum^{\infty} _{n=1}{\frac{\alpha_n}{s^n}}\equiv \Delta^{ s} _{\alpha_1\alpha_2...\alpha_n...}=x,
\end{equation}
where $\alpha_n\in A_s \equiv\{0,1,\dots , s-1\}$.

\begin{definition}
 A representation of the form~\eqref{eq: s-adic-1} is called \emph{the $s$-adic representation} of $x$. The notation  $\alpha_{n+1}(x)\alpha_{n+2}(x)...$ is a non-fixed tuple of digits in the $s$-adic representation of a number~$x$. 
\end{definition}

Notice that in the case of the $s$-adic representation there exists countable  set of numbers having two different $s$-adic representations. These numbers have the form
$$
\Delta^s _{\alpha_1\alpha_2...\alpha_{n-1}\alpha_n(0)}=\Delta^s _{\alpha_1\alpha_2...\alpha_{n-1}[\alpha_n-1](s-1)},~~~\alpha_n\ne 0, 
$$
and  are called \emph{s-adic rational}. Symbols in parentheses mean the period in the representation of a number. The other numbers in $[0,1]$ are called \emph{$s$-adic irrational}.

\begin{definition}
Any number $x\in\left[-\frac{s}{s+1},\frac{1}{s+1}\right]$ can be represented by the following form that is  called \emph{the nega-$s$-adic representation}:
$$
x=\Delta^{-s} _{\alpha_1\alpha_2...\alpha_n...}\equiv \sum^{\infty} _{n=1}{\frac{\alpha_n}{(-s)^n}},
$$
where $\alpha_n\in A_s$.
\end{definition}
\begin{remark} The term ``nega" is used throughout the paper since corresponding encodings of real numbers  are numeral systems with a negative base.
\end{remark}

It is easy to see that the following relationship between the $s$-adic and nega-$s$-adic representations holds:
\begin{equation}
\label{eq: 1}
\Delta^{-s} _{\alpha_1\alpha_2...\alpha_n...}+\sum^{\infty} _{k=1}{\frac{s-1}{s^{2k-1}}}=\sum^{\infty} _{k=1}{\frac{\alpha_{2k}}{s^{2k}}}+\sum^{\infty} _{k=1}{\frac{s-1-\alpha_{2k-1}}{s^{2k-1}}}\equiv\Delta^{s} _{[s-1-\alpha_1]\alpha_2...[s-1-\alpha_{2k-1}]\alpha_{2k}...}.
\end{equation}

Let $\mathbb N_B$ be a fixed subset of positive integers and
$$
\rho_n=\begin{cases}
1&\text{if $n\in \mathbb N_B$}\\
2&\text{if  $n\notin \mathbb N_B$.}
\end{cases}
$$

 Let us consider \emph{ the quasi-nega-$s$-adic representation} described in \cite{S.Serbenyuk}:
\begin{equation}
\label{eq: 2}
x=\Delta^{(\pm s, \mathbb N_B)} _{\alpha_1\alpha_2...\alpha_n...}\equiv \sum^{\infty} _{n=1} {\frac{(-1)^{\rho_n}\alpha_n}{s^n}}.
\end{equation}
Here $\alpha_n\in A_s$ and $x\in[a^{'} _{0}, a^{''} _{0}]$, where $a^{''} _{0}=1+a^{'} _{0}$ and
$$
a^{'} _{0}=-\sum^{\infty} _{n=1}{\frac{s-1}{s^{b_n}}}.
$$

\begin{theorem}{\cite{S.Serbenyuk}.}
Any number $x\in[a^{'} _{0}, a^{''} _{0}]$ can be represented  in form \eqref{eq: 2}.
\end{theorem}

One can note that if we consider $\Delta^{(\pm s, \mathbb N_B)} _{\alpha_1\alpha_2...\alpha_n...}$, then $ \Delta^{s} _{\alpha^{'} _1\alpha^{'} _2...\alpha^{'} _n...}\in [0,1]$, where
$$
\alpha^{'} _n=\begin{cases}
s-1-\alpha_n&\text{if $n\in \mathbb N_B$}\\
\alpha_n&\text{if  $n\notin \mathbb N_B$.}
\end{cases}
$$
That is, 
\begin{equation}
\label{eq: 3}
\Delta^{(\pm s, \mathbb N_B)} _{\alpha_1\alpha_2...\alpha_n...}+\sum^{\infty} _{n=1}{\frac{s-1}{s^{b_n}}}\equiv \Delta^{s} _{\alpha^{'} _1\alpha^{'} _2...\alpha^{'} _n...}.
\end{equation}
Also, we can write
$$
\Delta^{(\pm s, \mathbb N_B)} _{\alpha_1\alpha_2...\alpha_n...}\equiv \Delta^{s} _{\theta(\alpha_1)\theta(\alpha_2)...\theta(\alpha_n)...}-\sum^{\infty} _{n=1}{\frac{s-1}{s^{b_n}}},
$$
where $\theta(\alpha_n)=\alpha^{'} _n$.

One can generalize this idea and model representations  related with the $s$-adic representation by certain number converters (converters of combinations of numbers). There exist numeral systems constructed by certain mixings points in known numeral systems. In the present article,  the operator approach to modeling numeral systems is introduced. It is shown that there exist real number representations, that are modified $s$-adic representations, in which relations between digits in the old and new representations are defined by certain operators $\theta$ of changes of number combinations. In other words,  new encodings are obtained by  certain changes of s-adic number combinations by operators $\theta$ in the $s$-adic representations  of numbers.

Let us consider the following operators (converters of  $k$-digit combinations):
$$
\theta_{k,i} \left(\alpha_{km+1},\alpha_{km+2},\dots ,\alpha_{(m+1)k}\right)=\left(\beta_{km+1},\beta_{km+2},\dots ,\beta_{(m+1)k}\right),
$$
where numbers $k, i$ are fixed for an operator $\theta_{k,i}$, $m=0,1,\dots$, and $\theta_{k,i} (\alpha_1,\alpha_2,\dots ,\alpha_k)$ is  some  bijective correspondence such that the
set
$$
A_s ^k=\underbrace{A_s \times A_s \times \ldots\times A_s}_{k}.
$$
is its domain of definition and range of values. Also, here $\alpha_j\in A_s$ for $j=1,2, \dots, k$ and  $i =0, 1, \dots , (s^k)!-1$ since the number of all operators (converters) $\theta_{k,i}$ for a fixed number $k$ is equal to the number of all permutations of elements from the set of all samples ordered with reiterations of $k$ numbers from $A_s$.
That is, each combination $(\alpha_1,\alpha_2,\dots ,\alpha_k)$ of $k$ s-adic digits  is assigned to the single
combination $\theta_{k,i} (\alpha_1,\alpha_2,\dots ,\alpha_k)$ of $k$ s-adic digits, i.e.,
$$
\left(\beta_{1},\beta_{2},\dots ,\beta_{k}\right)=\theta_{k,i} \left(\alpha_{1},\alpha_{2},\dots ,\alpha_{k}\right),
$$
$$
\left(\beta_{k+1},\beta_{k+2},\dots ,\beta_{2k}\right)=\theta_{k,i} \left(\alpha_{k+1},\alpha_{k+2},\dots ,\alpha_{2k}\right),
$$
$$
\dots \dots \dots \dots \dots \dots \dots 
$$
$$
\left(\beta_{km+1},\beta_{km+2},\dots ,\beta_{(m+1)k}\right)=\theta_{k,i} \left(\alpha_{km+1},\alpha_{km+2},\dots ,\alpha_{(m+1)k}\right),
$$
$$
\dots \dots \dots \dots \dots \dots \dots 
$$

Assume that for any fixed numbers $k\in\mathbb N$ and $2\le s\in\mathbb N$ the operators $\theta_{k,0}$ and $\theta_{k,(s^k)!-1}$ are the following:
$$
\theta_{k,0} \left(\alpha_{km+1},\alpha_{km+2},\dots ,\alpha_{(m+1)k}\right)=\left(\alpha_{km+1},\alpha_{km+2},\dots ,\alpha_{(m+1)k}\right),
$$
$$
\theta_{k,(s^k)!-1} \left(\alpha_{km+1},\alpha_{km+2},\dots ,\alpha_{(m+1)k}\right) =\left(s-1-\alpha_{km+1},s-1-\alpha_{km+2},\dots ,s-1-\alpha_{(m+1)k}\right),
$$
where $m=0,1,2,3, \dots$.

Let us consider certain examples of $\theta_{k,i}$.
\begin{example}
Suppose that $s=k=2$. Let us consider the operator $\theta_{2,i}$ defined by the following table
\begin{center}
\begin{tabular}{|c|c|c|c|c|c|}
\hline
$\alpha_{2m+1}\alpha_{2(m+1)}$ & $ 00 $& $01$ & $10$ & $11$\\
\hline
$\beta_{2m+1}\beta_{2(m+1)}$ & $10$ & $11$ & $00$ & $01$\\
\hline
\end{tabular}
\end{center}
Here $m=0,1,\dots$, $\beta_{2m+1}\beta_{2(m+1)}=\theta_{2,i}(\alpha_{2m+1},\alpha_{2(m+1)})$, and $i\in\{0,1,\dots , 23\}$. 

For example, a number having a binary representation of the form $\Delta^2 _{1110011001(11)}$ will have the following representation in terms of a new modified representation:
$$
\Delta^{2^{'}} _{0100110011(01)}=\theta_{2,i}\left(\Delta^2 _{1110011001(11)}\right).
$$
That is, 
$$
f_{\theta_{2,i}}: ~~~~~~~~~~~~~\Delta^2 _{1110011001(11)} \to \Delta^{2^{'}} _{0100110011(01)}.
$$
Note that $\Delta^{2^{'}} _{\beta_1\beta_2...\beta_n...}$ is a formal denotation of $\Delta^2 _{\theta_{2,i}(\alpha_1,\alpha_2)\theta_{2,i}(\alpha_3,\alpha_4)...}$.
\end{example}

\begin{example}
Suppose that $s=7, k=1,$ and $\beta_n=\theta_{1,i}(\alpha_n)$, where  $i\in\{0,1,\dots , 7! -1\}$ and 
\begin{center}
\begin{tabular}{|c|c|c|c|c|c|c|c|}
\hline
$\alpha_n$ & $ 0 $ & $1$ & $2$ & $3$ & $4$ & $5$ & $6$\\
\hline
$\beta_n$ & $ 3 $ & $5$ & $6$ & $4$ & $0$ & $2$ & $1$\\
\hline
\end{tabular}
\end{center}
It is easy to see that for an operator defined by the last-mentioned table and for all $n, t\in\mathbb N$, $\alpha_n\in \{0,1,\dots , 6\}$, the equality
$$
\underbrace{\theta_{1,i} \circ \theta_{1,i} \circ \cdots\circ \theta_{1,i}(\alpha_n)}_{12t}=\alpha_n
$$
holds  since the following system of equalities is true for this operator:
$$
\begin{cases}
{\theta_{1,i} \circ \theta_{1,i} \circ \theta_{1,i}(\alpha)}=\alpha&\text{whenever $\alpha\in\{0,3,4\}$}\\
{\theta_{1,i} \circ \theta_{1,i} \circ \theta_{1,i}\circ \theta_{1,i}(\alpha)}=\alpha&\text{whenever $\alpha\in\{1,2,5,6\}$.}
\end{cases}
$$
For example, 
$$
f_{\theta_{1,i}}: \Delta^7 _{3455142(1)}\to \Delta^{7^{'}} _{4022506(5)}.
$$
\end{example}

\section{General definitions}
\label{section 2a}

Let $s>1$ be a fixed positive integer and $(k_n)$ be a fixed sequence, where $k_n\in\mathbb N$. Let us consider the following matrix of operators
\begin{equation}
\label{eq: operator matrix}
\Theta_{s, (k_n)}\equiv
\begin{pmatrix}
\theta_{k_1,0}&  \theta_{k_1,1} &\ldots & \theta_{k_1,i_{j}}& \ldots& \ldots& \theta_{k_1,(s^{k_1})!-1}\\
\theta_{k_2,0}&  \theta_{k_2,1} &\ldots & \theta_{k_2,i_{j}}& \ldots&\ldots&\ldots &\theta_{k_2,(s^{k_2})!-1}\\
\vdots& \vdots &\ddots & \vdots&\ldots&\ldots&\ldots&\ldots&\ldots\\
\theta_{k_n,0}& \theta_{k_n,1} &\ldots & \theta_{k_n,i_{j}} &\ldots& \theta_{k_n,(s^{k_n})!-1}&\\
\dots  &   \dots &  \dots        & \dots &         \dots               &  \dots &  \dots &  \dots &  \dots           
\end{pmatrix}.
\end{equation}
Here $i_{j}=0, 1, \dots , (s^{k_n})!-1$. Clearly, it is possible that in this matrix the  number of all rows is equal to infinity (since $(k_n)$ can be an infinite sequence), and the numbers of elements in different rows can be different.

From matrix \eqref{eq: operator matrix}, a sequence $\left(\theta_{k_n,i_{k_n}}\right)$ of operators  is generated by selecting an  arbitrary operator $\theta_{k_n,i_{k_n}}$ from the row $n$ for all $n=1, 2, 3,  \dots $. That is,  $\theta_{k_n,i_{k_n}}$ is an operator from   the row $n$ and the column $i_{k_n}$.

So, we obtain a representation $\Delta^{\left(s,\left(\theta_{k_n,i_{k_n}}\right)\right)} _{\beta_1\beta_2...\beta_n...}$ related with the $s$-adic representation by the following rule:
\begin{equation}
\label{eq: 4}
\Delta^{\left(s,\left(\theta_{k_n,i_{k_n}}\right)\right)} _{\beta_1\beta_2...\beta_n...}\equiv \Delta^s _{\theta_{k_1,i_{k_1}}(\alpha_1,\dots , \alpha_{k_1})\theta_{k_2,i_{k_2}}(\alpha_{k_1+1},\dots , \alpha_{k_1+k_2})...\theta_{k_n,i_{k_n}}(\alpha_{k_1+k_2+\dots+k_{n-1}+1},\dots , \alpha_{k_1+\dots+k_n})...}.
\end{equation}
\begin{definition}
A representation $\Delta^{\left(s,\left(\theta_{k_n,i_{k_n}}\right)\right)} _{\beta_1\beta_2...\beta_n...}$ whose relation with the $s$-adic representation is described by equality \eqref{eq: 4}, is called \emph{the pseudo-s-adic representation of numbers from $[0,1]$}.
\end{definition}

\begin{example}
Suppose we have the representation of real numbers from $[0,1]$ by positive Cantor series, i.e.,
$$
x=\Delta^Q _{\varepsilon_1\varepsilon_2...\varepsilon_n...}\equiv\sum^{\infty} _{n=1}{\frac{\varepsilon_n}{q_1q_2\cdots q_n}}, 
$$
where $Q=(q_n)$ is a fixed sequence of positive integers, $q_n\ge 2$ for all $n\in\mathbb N$, and $\varepsilon_n\in\{0,1,\dots , q_n-~1\}$. In this case,  matrix \eqref{eq: operator matrix} is defined as:
\begin{equation*}
\Theta_{Q, (k_n)}\equiv
\begin{pmatrix}
\theta_{k_1,0}&  \theta_{k_1,1} &\ldots & \theta_{k_1,i_{j}}& \ldots& \ldots& \theta_{k_1, q_1q_2\cdots q_{k_1}-1}\\
\theta_{k_2,0}&  \theta_{k_2,1} &\ldots & \theta_{k_2,i_{j}}& \ldots&\ldots&\ldots &\theta_{k_2, q_{k_1+1}q_{k_1+2}\cdots q_{k_2}-1}\\
\vdots& \vdots &\ddots & \vdots&\ldots&\ldots&\ldots&\ldots&\ldots\\
\theta_{k_n,0}& \theta_{k_n,1} &\ldots & \theta_{k_n,i_{j}} &\ldots& \theta_{k_n,q_{k_{n-1}+1}\cdots q_{k_n}-1}&\\
\dots  &   \dots &  \dots        & \dots &         \dots               &  \dots &  \dots &  \dots &  \dots           
\end{pmatrix},
\end{equation*}
where $k_n\in\mathbb N$ and $i_{j}=0, 1, \dots , q_1q_2\cdots q_{k_n}-1$. Also, here $\theta_{k_n,0}$ is the identical operator   and
$$
\theta_{k_n,q_{k_{n-1}+1}\cdots q_{k_n}-1}\left(\varepsilon_{k_{n-1}+1},\varepsilon_{k_{n-1}+2},\dots , \varepsilon_{k_{n}}\right)=\left(q_{k_{n-1}+1}-1-\varepsilon_{k_{n-1}+1},\dots , q_{k_{n}}-1-\varepsilon_{k_{n}}\right).
$$
\end{example}

\section{One example of the pseudo-ternary representation}
\label{section 3}

Suppose   $s=3$ and  $k_n=1$ for all $n\in \mathbb N$. Then there exist $6=3!$ operators $\theta_{1,i_j}$ (see the following table).
\begin{center}
\begin{tabular}{|c|c|c|c|c|}
\hline
  $$ & $ \beta_{1,i_j}(0)$ &$  \beta_{1,i_j}(1) $ & $ \beta_{1,i_j}(2)$\\
\hline
$\alpha_n$ &$ $ 0 &$ 1 $ & $2$\\
\hline
$\theta_{1,0}  (\alpha_n) $ &$0$ & $1$ & $2$\\
\hline
$\theta_{1,1} (\alpha_n) $ &$0$ & $2$ & $1$\\
\hline
$\theta_{1,2} (\alpha_n) $ &$1$ & $0$ & $2$\\
\hline
$\theta_{1,3} (\alpha_n) $ &$1$ & $2$ & $0$\\
\hline
$\theta_{1,4} (\alpha_n) $ &$2$ & $0$ & $1$\\
\hline
$\theta_{1,5} (\alpha_n) $ &$2$ & $1$ & $0$\\
\hline
\end{tabular}
\end{center}

Let us consider the case of $\theta_{1,1}$. To simplify mathematical notations, let us denote $\theta_{1,1}(\alpha_n)$ as $\theta (\alpha_n)$. Let us investigate the pseudo-ternary representation (or \emph{the $3^{'}$-representation}) generated by the operator $\theta (\alpha_n)$, i.e., for $x \in [0, 1]$ we have
\begin{equation}
\label{eq: 3'}
 x=\Delta^{(3, \theta)} _{\beta_1\beta_2...\beta_n...}=\Delta^{3^{'}} _{\beta_1\beta_2...\beta_n...}\equiv\Delta^{3} _{\theta(\alpha_1)\theta(\alpha_2)...\theta(\alpha_n)...}\equiv\sum^{\infty} _{n=1}{\frac{\theta(\alpha_n)}{3^n}}.
\end{equation}
Note that 
$$
\Delta^{3} _{\alpha_1\alpha_2..\alpha_n...}\equiv \Delta^{3^{'}} _{\theta(\beta_1)\theta(\beta_2)...\theta(\beta_n)...}
$$
since $\theta\left(\theta(\alpha_n)\right)=\theta\left(\beta_n\right)=\alpha_n$ holds for any $n\in\mathbb N$.

 One can represent the operator $\theta$ by the Lagrange polynomial, i.e.,
$$
\theta(\alpha_n)=\frac{7\alpha_n-3\alpha^2 _n}{2}.
$$
Let us check the correctness of this representation. Really, in terms of the Lagrange polynomial, we obtain 
$$
\theta\left(\theta(\alpha_n)\right)\ne \frac{7\alpha_n-3\alpha^2 _n}{2}.
$$
So we shall not  use the Lagrange polynomial for representing such operators.

\begin{definition}In the case of the $3^{'}$-representation, we obtain
$$
\Delta^{3^{'}} _{\alpha_1\alpha_2...\alpha_{n-1}\alpha_n(0)}=\Delta^{3^{'}} _{\alpha_1\alpha_2...\alpha_{n-1}[\alpha_n-1](2)},~~~\alpha_n\ne 0.
$$
Such numbers are called \emph{$3^{'}$-rational} and the other numbers having the unique $3^{'}$-representation are called \emph{$3^{'}$-irrational}.
\end{definition}

\begin{lemma}
\label{lm: lemma 2-problems}
Any number $x\in[0,1]$ can be represented in form \eqref{eq: 3'}.

Every $3^{'}$-irrational number has the unique $3^{'}$-representation.
Every $3^{'}$-rational number has two different $3^{'}$-representations.
\end{lemma}
\begin{proof} Let us consider the function  
\begin{equation}
\label{eq: function f}
y=f(x)=f\left(\Delta^3 _{\alpha_1\alpha_2...\alpha_n...}\right)=\Delta^{3} _{\theta(\alpha_1)\theta(\alpha_2)...\theta(\alpha_n)...}. 
\end{equation}
Here we shall not consider numbers whose  ternary representation has the period $(2)$ (without the number $1$).  

This function and a class of functions containing this function were investigated in \cite{{S. Serbenyuk abstract 6}, {S. Serbenyuk preprint2},  Symon12(2), {S. Serbenyuk functions with complicated local structure 2013},{S. Serbenyuk abstract 7}, {S. Serbenyuk abstract 8}}. 

Let us note the main properties of this function.
\begin{theorem}{\cite{{S. Serbenyuk abstract 6}, Symon12(2), {S. Serbenyuk functions with complicated local structure 2013},  {S. Serbenyuk preprint2}}}.
The function $f$ has the following properties:
\begin{itemize}
\item $$
[0,1]\stackrel{f}{\rightarrow} \left([0,1] \setminus \{\Delta^3 _{\alpha_1\alpha_2...\alpha_n111...}\}\right) \cup\left\{\frac{1}{2}\right\};
$$
\item the point $x_0=0$ is the unique invariant point of the function  $f$;

\item $f$ is continuous at ternary-irrational points, and ternary-rational points are points of discontinuity of the function; furthermore, a ternary-rational point $x_0=\Delta^3 _{\alpha_1\alpha_2...\alpha_n000...} $ is a  point of discontinuity $\frac{1}{2\cdot 3^{n-1}}$ whenever $\alpha_n=1$, and is a point of discontinuity $\left(-\frac{1}{2\cdot3^{n-1}}\right)$ whenever $\alpha_n=2$;

\item  the function $f$ is not monotonic on the domain of definition;

\item the function $f$ is non-differentiable;

\item the Hausdorff-Besicovitch dimension of the graph of $f$  is equal to 1;

\item the Lebesgue integral of the function $f$   is equal to $\frac{1}{2}$.
\end{itemize}
\end{theorem}

Our statement follows from properties of the function and the well-posedness of the definition of this function.
\end{proof}

\begin{definition}
A set of the form
$$
\Delta^{3^{'}} _{c_1c_2...c_n}\equiv\left\{x: x=\Delta^{3^{'}} _{c_1c_2...c_n\beta_{n+1}\beta_{n+2}\beta_{n+3}...}, \beta_t\in\{0,1,2\}, t>n \right\},
$$
where $c_1,c_2,\dots , c_n$ is an
ordered tuple of integers such that $c_j\in\{0,1,2\}$ for $j=\overline{1,n}$, is called \emph{a cylinder $\Delta^{3^{'}} _{c_1c_2...c_n}$ of rank $n$ with base $c_1c_2\ldots c_n$}.
 \end{definition}
\begin{lemma}
\label{lm: cylinder's properties}
Cylinders $\Delta^{3^{'}} _{c_1c_2...c_n}$ have the following properties:
\begin{enumerate}
\item A cylinder $\Delta^{3^{'}} _{c_1c_2...c_n}$ is a closed interval and
$$
\Delta^{3^{'}} _{c_1c_2...c_n}=\left[\Delta^{3^{'}} _{c_1c_2...c_n(0)},\Delta^{3^{'}} _{c_1c_2...c_n(2)}\right].
$$
\item 
$$
\left|\Delta^{3^{'}} _{c_1c_2...c_n}\right|=\frac{1}{3^n}.
$$
\item
$$
\Delta^{3^{'}} _{c_1c_2...c_nc}\subset \Delta^{3^{'}} _{c_1c_2...c_n}.
$$
\item
$$
\Delta^{3^{'}} _{c_1c_2...c_n}=\bigcup^2 _{c=0}{\Delta^{3^{'}} _{c_1c_2...c_nc}}.
$$
\item Cylinders $\Delta^{3^{'}} _{c_1c_2...c_n}$ are  left-to-right situated.
\end{enumerate}
\end{lemma}
\begin{proof}
Let us consider  map~\eqref{eq: function f} (i.e., $\theta(\alpha)=\theta_{1,1}(\alpha)$):
$$
f\left(\Delta^3 _{\alpha_1\alpha_2...\alpha_n...}\right)=\Delta^3 _{\theta(\alpha_1)\theta(\alpha_2)...\theta(\alpha_n)...}.
$$
\begin{enumerate} 
\item It is known that a cylinder $\Delta^{3} _{b_1b_2...b_n}$ is a closed interval and $f$ (see \cite{{S. Serbenyuk functions with complicated local structure 2013},   {S. Serbenyuk preprint2}}) is continuous at ternary irrational  points, and
ternary rational  points are points of discontinuity of this function.

Let us prove that a cylinder $\Delta^{3^{'}} _{c_1c_2...c_n}$ is a closed interval. 

Suppose that $x\in \Delta^{3^{'}} _{c_1c_2...c_n}$. That is,
$$
x=\sum^{n} _{j=1}{\frac{c_j}{3^j}}+\frac{1}{3^n}\sum^{\infty} _{l=1}{\frac{\beta_{n+l}}{3^l}}, ~~~\beta_{n+l}\in\{0,1,2\}.
$$
Whence
$$
x^{'}=\sum^{n} _{j=1}{\frac{c_j}{3^j}}\le x\le \sum^{n} _{j=1}{\frac{c_j}{3^j}}+\frac{1}{3^n}=x^{''}.
$$
So $x\in\left[x^{'}, x^{''}\right]\supseteq\Delta^{3^{'}} _{c_1c_2...c_n}$. Since
$$
x^{'}=\sum^{n} _{j=1}{\frac{c_j}{3^j}}+\frac{1}{3^n}\inf \sum^{\infty} _{l=1}{\frac{\beta_{n+l}}{3^l}}
$$
and
$$
x^{''}=\sum^{n} _{j=1}{\frac{c_j}{3^j}}+\frac{1}{3^n}\sup \sum^{\infty} _{l=1}{\frac{\beta_{n+l}}{3^l}},
$$
we obtain $x, x^{'}, x^{''}\in \Delta^{3^{'}} _{c_1c_2...c_n}$.
Let us prove that conditions $f(x_0)\in f\left(\Delta^{3} _{b_1b_2...b_n}\right)$ and $f\left(\Delta^{3} _{b_1b_2...b_nb}\right)\subset f\left(\Delta^{3} _{b_1b_2...b_n}\right)$ hold for any $x_0\in \Delta^{3} _{b_1b_2...b_n}$.

\item It is easy to see that
\begin{equation}
\label{eq: 8}
\begin{split}
f(x_0)&=f\left(\Delta^{3} _{b_1b_2...b_n\alpha_{n+1}(x_0)\alpha_{n+2}(x_0)...}\right)\\
&=\Delta^{3} _{\theta(b_1)\theta(b_2)...\theta(b_n)\theta(\alpha_{n+1})\theta(\alpha_{n+2})...}\\
&=\Delta^{3^{'}} _{c_1c_2...c_n\beta_{n+1}\beta_{n+2}...}\in \Delta^{3^{'}} _{c_1c_2...c_n}.
\end{split}
\end{equation}
Here  $\theta(b_j)=c_j$ for all $j=1, 2, \dots , n$ and $ c = \theta (b)$. Note that
$$
f\left(\inf \Delta^{3} _{b_1b_2...b_n}\right)=\inf{\Delta^{3^{'}} _{c_1c_2...c_n}}=\Delta^{3^{'}} _{c_1c_2...c_n(0)}.
$$
However,
$$
f\left(\sup \Delta^{3} _{b_1b_2...b_n}\right)\ne f\left(\Delta^{3} _{b_1b_2...b_n(1)}\right)= \sup{\Delta^{3^{'}} _{c_1c_2...c_n}}=\Delta^{3^{'}} _{c_1c_2...c_n(2)}.
$$

\item From relationship \eqref{eq: 8} and the following equalities
$$
f\left(\inf \Delta^{3} _{b_1b_2...b_nb}\right)=\Delta^{3^{'}} _{c_1c_2...c_nc(0)}\in \Delta^{3^{'}} _{c_1c_2...c_nc},
$$
$$
f\left(\sup \Delta^{3} _{b_1b_2...b_nb}\right)=\Delta^{3^{'}} _{c_1c_2...c_nc(1)}\in \Delta^{3^{'}} _{c_1c_2...c_nc},
$$
where
$$
f\left(\inf \Delta^{3} _{b_1b_2...b_nb}\right)=\inf{\Delta^{3^{'}} _{c_1c_2...c_nc}}
$$
and
$$
f\left(\sup \Delta^{3} _{b_1b_2...b_nb}\right)<\sup{\Delta^{3^{'}} _{c_1c_2...c_nc}},
$$
it follows that
$$
\inf{\Delta^{3^{'}} _{c_1c_2...c_nc}}\ge \inf{\Delta^{3^{'}} _{c_1c_2...c_n}},
$$
$$
\sup{\Delta^{3^{'}} _{c_1c_2...c_nc}}\le \sup{\Delta^{3^{'}} _{c_1c_2...c_n}}.
$$

\item So,
$$
\Delta^{3^{'}} _{c_1c_2...c_nc}\subset \Delta^{3^{'}} _{c_1c_2...c_n}=\bigcup^2 _{j=1}{\Delta^{3^{'}} _{c_1c_2...c_nj}}.
$$

\item In order to prove that the 5th property is true we observe that
$$
\inf\Delta^{3^{'}} _{c_1c_2...c_{n-1}2}-\sup\Delta^{3^{'}} _{c_1c_2...c_{n-1}1}=\Delta^{3^{'}} _{c_1c_2...c_{n-1}2(0)}-\Delta^{3^{'}} _{c_1c_2...c_{n-1}1(2)}=0,
$$
$$
\inf\Delta^{3^{'}} _{c_1c_2...c_{n-1}1}-\sup\Delta^{3^{'}} _{c_1c_2...c_{n-1}0}=\Delta^{3^{'}} _{c_1c_2...c_{n-1}1(0)}-\Delta^{3^{'}} _{c_1c_2...c_{n-1}0(2)}=0.
$$
\end{enumerate}
\end{proof}

\begin{lemma}
The map $f$ does not preserve the distance between points.
\end{lemma}
\begin{proof}
Let us find points $x_1, x_2\in[0,1]$ such that  $(f(x_2)-f(x_1)\ne x_2-x_1$. The statement follows from the existence of jump discontinuities of $f$. Indeed, for example, if we consider $x_1=\Delta^3 _{1(0)}$ and $x_2=\Delta^3 _{11(0)}$, we have
$$
f(x_2)-f(x_1)=\Delta^3 _{22(0)}-\Delta^3 _{2(0)}=\frac 2 9\ne x_2-x_1=\frac 1 9.
$$

In the next section, this statement for a general case of the pseudo-binary representation will be proven in  detail. 
\end{proof}

\begin{theorem} The map
$$
f: \Delta^3 _{\alpha_1\alpha_2...\alpha_n...} \to\Delta^3 _{\theta(\alpha_1)\theta(\alpha_2)...\theta(\alpha_n)...}
$$
preserves the Lebesgue measure of an arbitrary interval (segment).
\end{theorem}
\begin{proof}
Let a segment $[a^{'},a^{''}]\subset [0,1]$ be a certain cylinder $\Delta^{3} _{b_1b_2...b_n}$. Then from Lemma \ref{lm: cylinder's properties} it follows that the Lebesgue measure of the image and the preimage are equal:
$$
\left|\Delta^3 _{b_1b_2...b_n}\right|=\left|\Delta^{3^{'}} _{c_1c_2...c_n}\right|=\frac{1}{3^n}.
$$

Let $[a^{'},a^{''}]\subset [0,1]$ be an interval that is not a certain cylinder $\Delta^{3} _{b_1b_2...b_n}$. Then
there exists $\varepsilon$-covering of this segment by cylinders  $\Delta^{3} _{b_1b_2...b_k}$ of rank $k$ such that
$$
[a^{'},a^{''}]\subseteq\bigcup_k \Delta^{3} _{b_1b_2...b_k}
$$
and
$$
\lim_{k\to\infty}{\left|\bigcup_k \Delta^{3} _{b_1b_2...b_k}\right|}=\left|[a^{'},a^{''}]\right|.
$$
Also, there exists $\varepsilon$-covering of this segment by cylinders  $\Delta^{3} _{b_1b_2...b_m}$ of rank $m$ such that
$$
[a^{'},a^{''}]\supseteq\bigcup_m \Delta^{3} _{b_1b_2...b_m}
$$
and
$$
\lim_{m\to\infty}{\left|\bigcup_m \Delta^{3} _{b_1b_2...b_m}\right|}=\left|[a^{'},a^{''}]\right|.
$$
Here $\varepsilon$ is an arbitrary small positive number. That is,
$$
\left|\bigcup_m \Delta^{3} _{b_1b_2...b_m}\right|+\varepsilon\le\left|[a^{'},a^{''}]\right|\le\left|\bigcup_k \Delta^{3} _{b_1b_2...b_k}\right|-\varepsilon.
$$
Since,
$$
\left|\bigcup_k \Delta^{3} _{b_1b_2...b_k}\right|=\sum_k{\left| \Delta^{3} _{b_1b_2...b_k}\right|},~~~\left|\bigcup_m \Delta^{3} _{b_1b_2...b_m}\right|=\sum_m{\left| \Delta^{3} _{b_1b_2...b_m}\right|},
$$
we have
$$
\lim_{m\to\infty}{\sum_m{\left| \Delta^{3} _{b_1b_2...b_m}\right|}}=\left|[a^{'},a^{''}]\right|=\lim_{k\to\infty}{\sum_k{\left| \Delta^{3} _{b_1b_2...b_k}\right|}}.
$$
From Lemma  \ref{lm: cylinder's properties} the condition
$$
\left|\Delta^{3} _{b_1b_2...b_n}\right|=\left|f\left(\Delta^{3} _{b_1b_2...b_n}\right)\right|
$$
 holds for any integer  $n \ge 1$, we have 
$$
\left|[a^{'},a^{''}]\right|=\lim_{m\to\infty}{\sum_m{\left| f\left(\Delta^{3} _{b_1b_2...b_m}\right)\right|}}=\lim_{k\to\infty}{\sum_k{\left|f\left(\Delta^{3} _{b_1b_2...b_k}\right|\right)}}=\left|f\left([a^{'},a^{''}]\right)\right|.
$$
\end{proof}

\section{The pseudo-binary representations}
\label{section 3a} 

Note that the case when $s=2$ for the  pseudo-$s$-adic representation is interesting for consideration since this representation (the pseudo-binary representation) can be used in computer science by analogy with the classical binary representation. It can be useful for coding  information, protection of information, providing computer protection, etc. Certain theoretic aspects of such applications will be discussed in the further papers of the author of the present article.

Let us remark that in our case there exist only two one-digit converters (operators), i.e.,
$$
\theta_{1,0}(\alpha_n)=\theta_0(\alpha_n)=\begin{cases}
0&\text{if $\alpha_n=0$}\\
1&\text{if $\alpha_n=1$}
\end{cases}
$$
and
$$
\theta_{1,1}(\alpha_n)=\theta_1(\alpha_n)=\begin{cases}
1&\text{if $\alpha_n=0$}\\
0&\text{if $\alpha_n=1$}.
\end{cases}
$$
In other words, $\theta_0(\alpha_n)=\alpha_n$ and $\theta_1(\alpha_n)=1-\alpha_n$.

Let us denote
$$
\theta_{k_n,0} (\alpha_{k_1+\dots +k_{n-1}+1},\dots , \alpha_{k_1+\dots +k_n})=(\alpha_{k_1+\dots +k_{n-1}+1},\dots , \alpha_{k_1+\dots +k_n})
$$
and
$$
\theta_{k_n,(2^{k_n})!-1} (\alpha_{k_1+\dots +k_{n-1}+1},\dots , \alpha_{k_1+\dots +k_n})=(1-\alpha_{k_1+\dots +k_{n-1}+1},\dots , 1-\alpha_{k_1+\dots +k_n}),
$$
where $k_{n-1}=0$ for $n=1$.

Let $(k_n)$ be a fixed sequence of positive integers. Then we obtain the matrix
\begin{equation}
\label{eq: operator matrix2}
\Theta_{2, (k_n)}=
\begin{pmatrix}
\theta_{k_1,0}&  \theta_{k_1,1} &\ldots & \theta_{k_1,i_{j}}& \ldots& \ldots& \theta_{k_1,(2^{k_1})!-1}\\
\theta_{k_2,0}&  \theta_{k_2,1} &\ldots & \theta_{k_2,i_{j}}& \ldots&\ldots&\ldots &\theta_{k_2,(2^{k_2})!-1}\\
\vdots& \vdots &\ddots & \vdots&\ldots&\ldots&\ldots&\ldots&\ldots\\
\theta_{k_n,0}& \theta_{k_n,1} &\ldots & \theta_{k_n,i_{j}} &\ldots& \theta_{k_n,(2^{k_n})!-1}&\\
\dots  &   \dots &  \dots        & \dots &         \dots               &  \dots &  \dots &  \dots &  \dots           
\end{pmatrix}.
\end{equation}
Here for a fixed number $k_n$ there exist $(2^{k_n})!$ $k_n$-digit converters (operators) and $i_{j}=0,1, \dots , (s^{k_n})!-~1$. In this matrix, a number of rows can  equal to infinity (whenever $(k_n)$ is an infinite sequence) and the number of elements in the $n$-th row equals  $(2^{k_n})!$.

Let us generate a sequence $(\theta_{k_n,i_{n}})$ by selecting an arbitrary unique  operator $\theta_{k_n, i_{n}}$ from the fixed row $n$ (in a column $i_n$) for all $n \in \mathbb  N$.

 Let us investigate the representation of real numbers from $[0,1]$ that is related with the binary representation by the following rule:
\begin{equation}
\label{eq: definition of pseudo-binary}
y=\Delta^{\left(2, (\theta_{k_n, i_n})\right)} _{\beta_1\beta_2...\beta_n...}\equiv\Delta^2 _{\theta_{k_1,i_1} (\alpha_{1}, \alpha_2,...,\alpha_{k_1})...\theta_{k_n,i_n} (\alpha_{k_1+...+k_{n-1}+1}, ...,\alpha_{k_1+...+k_n})...}\equiv\sum^{\infty} _{j=1}{\frac{\beta_j}{2^j}},
\end{equation}
where
$$
\left(\beta_{1},\beta_{2},\dots ,\beta_{k_1}\right)=\theta_{k_1,i_1} \left(\alpha_{1},\alpha_{2},\dots ,\alpha_{k_1}\right),
$$
$$
\left(\beta_{k_1+1},\beta_{k_1+2},\dots ,\beta_{k_1+k_2}\right)=\theta_{k_2,i_2} \left(\alpha_{k_1+1},\alpha_{k_1+2},\dots ,\alpha_{k_1+k_2}\right),
$$
$$
\dots \dots \dots \dots \dots \dots \dots 
$$
$$
\left(\beta_{k_1+\dots+k_{n-1}+1},\dots ,\beta_{k_1+\dots+k_{n-1}+k_n}\right)=\theta_{k_n,i_n} \left(\alpha_{k_1+\dots+k_{n-1}+1},\dots ,\alpha_{k_1+\dots+k_{n-1}+k_n}\right),
$$
$$
\dots \dots \dots \dots \dots \dots \dots 
$$
That is, for any $ x=\Delta^2 _{\alpha_1\alpha_2...\alpha_n...}\in [0,1]$, the relationship between the binary and pseudo-binary representations is 
\begin{equation}
\label{eq: 11}
f_{2^{'}}:~~~~~~~~~~~ x=\Delta^2 _{\alpha_1\alpha_2...\alpha_n...} \to \Delta^{\left(2, (\theta_{k_n, i_n})\right)} _{\beta_1\beta_2...\beta_n...}=f_{2^{'}}(x)=y.
\end{equation}

The following condition,  which depends on the definition of $\theta_{k_n, i_n}$, where $k_n, i_n$ are   fixed numbers, holds for $1\le t\le 2^{k_n}$:
$$
\underbrace{\theta_{k_n,i_n} \circ \theta_{k_n,i_n} \circ \ldots \circ \theta_{k_n,i_n}}_{t}(\alpha_{k_1+\dots +k_{n-1}+1},\dots , \alpha_{k_1+\dots +k_n})=(\alpha_{k_1+\dots +k_{n-1}+1},\dots , \alpha_{k_1+\dots +k_n})
$$

\begin{definition} A number $x\in[0,1]$ of the form
$$
x=\Delta^2 _{\alpha_1\alpha_2...\alpha_{n-1}\alpha_n(0)}=\Delta^2 _{\alpha_1\alpha_2...\alpha_{n-1}[\alpha_n-1](1)}
$$
is called \emph{binary rational}. The other numbers have the unique binary representation and are called \emph{binary irrational}. Also, numbers of the form
$$
\Delta^{\left(2, (\theta_{k_n, i_n})\right)} _{\beta_1\beta_2...\beta_{n-1}\beta_n(0)}=\Delta^{\left(2, (\theta_{k_n, i_n})\right)} _{\beta_1\beta_2...\beta_{n-1}[\beta_n-1](1)}
$$
are called \emph{pseudo-binary rational}.  The set of all pseudo-binary rational numbers is a countable set. The other numbers are called \emph{pseudo-binary irrational}.
\end{definition}
\begin{theorem}
Any number $x\in[0,1]$ can be represented by form \eqref{eq: definition of pseudo-binary}.
\end{theorem}
\begin{proof}
Suppose that $(\theta_{k_n,i_n})$ is a fixed sequence of operators (converters) from matrix \eqref{eq: operator matrix2}. 

The value of any operator $\theta_{k_n,i_n}$, $n=1,2,\dots ,$ is defined by the table
\begin{center}
\begin{tabular}{|c|c|}
\hline
$\alpha_{k_1+\dots +k_{n-1}+1} \ldots  \alpha_{k_1+\dots +k_n}$  &$\beta_{k_1+\dots +k_{n-1}+1} \ldots  \beta_{k_1+\dots +k_n}$\\
\hline
$ \underbrace{0 0\ldots 00}_{k_n}$ &$\theta_{k_n,i_n} ( \underbrace{0 0\ldots 00}_{k_n})$  \\
\hline
$ \underbrace{0 0\ldots 01}_{k_n}$ &$\theta_{k_n,i_n} ( \underbrace{0 0\ldots 01}_{k_n})$  \\
\hline
$\dots $ &$\dots$  \\
\hline
$ \underbrace{11\ldots 11}_{k_n}$ &$\theta_{k_n,i_n} ( \underbrace{11\ldots 11}_{k_n})$  \\
\hline
\end{tabular}
\end{center}

Let us consider the map \eqref{eq: 11} and an arbitraty number $x_0$. 
\begin{remark}
Let us remark that images of the different binary representations of a binary rational number $x_0$ under the  map \eqref{eq: 11} can be two different numbers (see below). It is important the following: 
\begin{itemize}
\item If we investigate properties of  the map \eqref{eq: 11} as a function, then it is sufficiently do  not consider one of two different binary representations  for all binary rational numbers. Thus, binary rational points are points of discontinuity of the function. For the well-posedness of $f_{2^{'}}$, we shall not consider the binary representation, which has the period~$(1)$ (without the case of the representation of the number $1$).

\item If we investigate a numeral system modeling by the map \eqref{eq: 11}, then we consider the set of all values of this map and the map can not be a function. Here the existence of a preimage of any number from a certain interval is needed.  
\end{itemize}
\end{remark}

If a number $x_0$ is binary rational, i.e., 
$$
x_0= x_1\equiv \Delta^2 _{\alpha_1...\alpha_m(0)= \Delta^2 _{\alpha_1...\alpha_{m-1}[\alpha_m-1](1)}\equiv x_2},
$$
then for $x_0=\Delta^2 _{\alpha_1...\alpha_m(0)}\equiv x_1$ there exists a number $t$ such that
$$
y_1=f_{2^{'}}(x_1)=f_{2^{'}}\left(\Delta^2 _{\alpha_1...\alpha_m(0)}\right)
$$
$$
=\Delta^2 _{{\theta_{k_1,i_1}(\alpha_{1}, ...,\alpha_{k_1})...\theta_{k_{t-1},i_{t-1}}(\alpha_{k_1+...+k_{t-2}+1},...,\alpha_{k_1+...+k_{t-1}})\theta_{k_{t},i_{t}}(\underbrace{\alpha_{k_1+...+k_{t-1}+1}, ..., \alpha_m,0,...,0}_{k_t})\theta_{k_{t+1},i_{t+1}}(\underbrace{0,...,0}_{k_{t+1}}) ...}}.
$$
 Clearly, $y_1$ can be a pseudo-binary irrational number and it is  a pseudo-binary rational number whenever there exists $n_0$ such that for all $l\ge n_0$ we have
\begin{equation}
\label{eq: condition1}
\theta_{k_{t+l},i_{t+l}}(\underbrace{0, 0,\ldots , 0, 0}_{k_{t+l}})=(\underbrace{0, 0, \ldots , 0, 0}_{k_{t+l}})
\end{equation}
or
\begin{equation}
\label{eq: condition2}
\theta_{k_{t+l},i_{t+l}}(\underbrace{0, 0,\ldots , 0, 0}_{k_{t+l}})=(\underbrace{1, 1, \ldots , 1, 1}_{k_{t+l}}).
\end{equation}

By analogy, we get for $x_2\equiv \Delta^2 _{\alpha_1...\alpha_{m-1}[\alpha_m-1](1)}=x_0$ the following:
$$
f_{2^{'}}(x_2)=f_{2^{'}}\left(\Delta^2 _{\alpha_1...\alpha_{m-1}[\alpha_m-1](1)}\right)
$$
$$
=\Delta^2 _{\theta_{k_1,i_1} (\alpha_{1}, ...,\alpha_{k_1})...\theta_{k_{t-1},i_{t-1}} (\alpha_{k_1+...+k_{t-2}+1},...,\alpha_{\omega})\theta_{k_{t},i_{t}}(\underbrace{\alpha_{\omega+1}, ...,\alpha_{m-1}, \alpha_m-1,1,...,1}_{k_t})\theta_{k_{t+1},i_{t+1}}(\underbrace{1,...,1}_{k_{t+1}}) ...}
$$
$=y_2,$  where $\omega=k_1+\dots+k_{t-1}$. Whence $y_2$ is a pseudo-binary rational number whenever there exists $n_0$ such that for all $l\ge n_0$ one of conditions \eqref{eq: condition1}, \eqref{eq: condition2} is true. 

Suppose $n_0$ is a some fixed number. Note that $y_1=y_2$ whenever for all $j=1,2,3, \dots , $ 

$$
\left[
\begin{aligned}
\left[
\begin{aligned}
\left\{
\begin{aligned}
\theta_{k_{t},i_{t}}(\underbrace{\alpha_{k_1+...+k_{t-1}+1}, \dots, \alpha_m,0,\dots ,0}_{k_t})&=(\beta_{k_1+...+k_{t-1}+1},\dots , \beta_m, \dots , \beta_{k_1+...+k_t})\\
\theta_{k_{t+j},i_{t+j}}(\underbrace{0, \dots , 0}_{k_{t+j}})  & =(\underbrace{0, \dots , 0}_{k_{t+j}})\\
\theta_{k_{t},i_{t}}(\underbrace{\alpha_{k_1+...+k_{t-1}+1}, \dots, \alpha_{m-1},\alpha_m-1,1,...,1}_{k_t})&=(\beta_{k_1+...+k_{t-1}+1},\dots , \beta_{k_1+...+k_t-1}, \beta_{k_1+...+k_t}-1)\\
\theta_{k_{t+j},i_{t+j}}(\underbrace{1, \dots , 1}_{k_{t+j}})  & =(\underbrace{1, \dots , 1}_{k_{t+j}})\\
\end{aligned}
\right.\\
\left\{
\begin{aligned}
\theta_{k_{t},i_{t}}(\underbrace{\alpha_{k_1+...+k_{t-1}+1}, \dots, \alpha_m,0,\dots ,0}_{k_t})&=(\beta_{k_1+...+k_{t-1}+1},\dots , \beta_{k_1+...+k_t-1}, \beta_{k_1+...+k_t}-1)\\
\theta_{k_{t+j},i_{t+j}}(\underbrace{0, \dots , 0}_{k_{t+j}})  & =(\underbrace{1, \dots , 1}_{k_{t+j}})\\
\theta_{k_{t},i_{t}}(\underbrace{\alpha_{k_1+...+k_{t-1}+1}, \dots, \alpha_{m-1},\alpha_m-1,1,...,1}_{k_t})&=(\beta_{k_1+...+k_{t-1}+1},\dots , \beta_m, \dots , \beta_{k_1+...+k_t})\\
\theta_{k_{t+j},i_{t+j}}(\underbrace{1, \dots , 1}_{k_{t+j}})  & =(\underbrace{0, \dots , 0}_{k_{t+j}})\\
\end{aligned}
\right.\\
\end{aligned}
\right.\\
\left[
\begin{aligned}
\left\{
\begin{aligned}
\theta_{k_{t+n_0},i_{t+n_0}}(\underbrace{0, \dots , 0}_{k_{t+n_0}})  & =(\beta_{k_1+...+k_{t+n_0-1}+1},\dots ,  \beta_{k_1+\dots +k_{t+n_0}})\\
\theta_{k_{t+l},i_{t+l}}(\underbrace{0, \dots , 0}_{k_{t+l}})  & =(\underbrace{0, \dots , 0}_{k_{t+l}})&\text{for all $l>n_0$}\\
\theta_{k_{t+n_0},i_{t+n_0}}(\underbrace{1, \dots , 1}_{k_{t+n_0}})  & =(\beta_{k_1+...+k_{t+n_0-1}+1},\dots ,  \beta_{k_1+\dots +k_{t+n_0}-1}, \beta_{k_1+\dots +k_{t+n_0}}-1)\\
\theta_{k_{t+l},i_{t+l}}(\underbrace{1, \dots , 1}_{k_{t+l}})  & =(\underbrace{1, \dots , 1}_{k_{t+l}})&\text{for all $l>n_0$}\\
\end{aligned}
\right.\\
\left\{
\begin{aligned}
\theta_{k_{t+n_0},i_{t+n_0}}(\underbrace{0, \dots , 0}_{k_{t+n_0}})  & =(\beta_{k_1+...+k_{t+n_0-1}+1},\dots ,  \beta_{k_1+\dots +k_{t+n_0}-1}, \beta_{k_1+\dots +k_{t+n_0}}-1)\\
\theta_{k_{t+l},i_{t+l}}(\underbrace{0, \dots , 0}_{k_{t+l}})  & =(\underbrace{1, \dots , 1}_{k_{t+l}})&\text{for all $l>n_0$}\\
\theta_{k_{t+n_0},i_{t+n_0}}(\underbrace{1, \dots , 1}_{k_{t+n_0}})  & =(\beta_{k_1+...+k_{t+n_0-1}+1},\dots ,  \beta_{k_1+\dots +k_{t+n_0}})\\
\theta_{k_{t+l},i_{t+l}}(\underbrace{1, \dots , 1}_{k_{t+l}})  & =(\underbrace{0, \dots , 0}_{k_{t+l}})&\text{for all $l>n_0$}.\\
\end{aligned}
\right.\\
\end{aligned}
\right.\\
\end{aligned}
\right.
$$

If $x_0=\Delta^2 _{\alpha_1\alpha_2...\alpha_n...}$ is a binary-irrational number, then $y_0=f_{2^{'}}(x_0)= \Delta^{\left(2, (\theta_{k_n, i_n})\right)} _{\beta_1\beta_2...\beta_n...}$ can be pseudo-binary rational (whenever its representation contains a period $(0)$ or $(1)$) or pseudo-binary irrational. It depends on the choice of the  sequence $(\theta_{k_n,i_n})$.

In the general case, let us choose any $x^{'}, x^{''}\in[0,1]$.

If $x^{'}\ne x^{''}$, then:
\begin{itemize}
\item $y^{'}=f_{2^{'}}(x^{'})=f_{2^{'}}(x^{''})=y^{''}$  whenever only $y^{'}=y^{''}$ is pseudo-binary rational;
\item in the other case, $y^{'}\ne y^{''}$ when $x^{'}, x^{''}$ are binary irrational numbers and $y^{'},y^{''}$ are pseudo-binary irrational numbers or different pseudo-binary rational numbers.
\end{itemize}

Suppose  $x^{'}= x^{''}$, i.e., $x^{'}, x^{''}$ are binary rational, since only  binary rational numbers have two different binary representations. Then  $f_{2^{'}}(x^{'})=f_{2^{'}}(x^{''})$ or $f_{2^{'}}(x^{'})\ne f_{2^{'}}(x^{''})$ by analogy.

Since
$$
\inf_{x\in [0,1]}{f_{2^{'}}(x)}=0, ~~~~~~~~~~\sup_{x\in [0,1]}{f_{2^{'}}(x)}=1,
$$
and  the sets of all binary rational and pseudo-binary rational numbers are countable sets, we obtain that the set of all numbers having the unique pseudo-binary representation is a set of full Lebesgue measure and the map $f_{2^{'}}$ determines a numeral system.
\end{proof}

\begin{corollary}
Any map $f_{2^{'}}$ is a bijective mapping on the set
$$
[0,1]\setminus \left(S_Q\cup S^{'} _Q\right),
$$
where $S_Q$  is the set of all binary rational points and 
$$
S^{'} _Q\equiv\left\{x: f_{2^{'}}(x) \text{is pseudo-binary rational}\right\}.
$$
\end{corollary}

\begin{corollary} The set
$$
\left\{x: x=f^{-1} _{2^{'}}(y)\right\}
$$
is an one- or two-element set for any $y\in[0,1]$.
\end{corollary}

\begin{lemma}
\label{lm: lemma-f2-continuous}
A map $f_{2^{'}}$ is continuous at binary irrational points.

According to a sequence $(\theta_{k_n, i_n})$, a certain binary rational point is a  point of discontinuity of $f_{2^{'}}$ or $f_{2^{'}}$ is continuous at this point. 
\end{lemma}
\begin{proof}
Let $x_0$ be a binary irrational number. Then there exists a positive integer $n_0\ge 2$ such that the following system of conditions holds for $x_0=\Delta^2 _{\alpha_1\alpha_2...\alpha_n...}$ and $x=\Delta^2 _{\gamma_1\gamma_2...\gamma_n...}$:
$$
\begin{cases}
\alpha_r=\gamma_r &\text{for $r=\overline{1, n_0-1}$}\\
\alpha_{n_0}\ne\gamma_{n_0}.
\end{cases}
$$
From this system, it follows that the conditions $n_0\to\infty$ and $x\to x_0$ are equivalent. In addition, for 
$$
f_{2^{'}}(x_0)=f_{2^{'}}\left(\Delta^2 _{\alpha_1\alpha_2...\alpha_n...}\right)=\Delta^{\left(2, (\theta_{k_n, i_n})\right)} _{\beta_1\beta_2...\beta_n...}
$$
and
$$
f_{2^{'}}(x)=f_{2^{'}}\left(\Delta^2 _{\gamma_1\gamma_2...\gamma_n...}\right)=\Delta^{\left(2, (\theta_{k_n, i_n})\right)} _{\delta_1\delta_2...\delta_n...},
$$
we get
$$
\left|f_{2^{'}} (x) - f_{2^{'}}(x_0)\right|=\left|\sum^{\infty} _{n=1} {\frac{\delta_n -\beta_n }{2^n}}\right|\le \sum^{\infty} _{l=m_0} {\frac{|\delta_l-\beta_l|}{2^l}}\le \sum^{\infty} _{l=m_0} {\frac{1}{2^l}}=\frac{1}{2^{m_0-1}}\to 0 ~\mbox{as} ~m_0 \to \infty,
$$
since for any $n\in\mathbb N$ the inequality $\frac{|\delta_n -\beta_n |}{2^n}\le \frac{1}{2^n}$ holds. Here $m_0 \ge 2$ is an integer such that $\delta_j(x)=\beta_j(x_0)$ for all $j=1, 2, \dots , m_0-1$.

So $f_{2^{'}}$ is continuous at binary irrational points.

In the case when $x_0$ is binary rational, i.e., 
$$
x_0=\Delta^2 _{\alpha_1...\alpha_m(0)}=\Delta^2 _{\alpha_1...\alpha_{m-1}[\alpha_m-1](1)},
$$ 
it is clear that 
$$
\lim_{x\to x^- _0}{f_{2^{'}}(x)}=f_{2^{'}}\left(\Delta^2 _{\alpha_1...\alpha_{m-1}[\alpha_m-1](1)}\right)
$$
and
$$
\lim_{x\to x^+ _0}{f_{2^{'}}(x)}=f_{2^{'}}\left(\Delta^2 _{\alpha_1...\alpha_{m-1}\alpha_m(0)}\right).
$$
Then:
\begin{itemize}
\item if $\lim_{x\to x^- _0}{f_{2^{'}}(x)}\ne\lim_{x\to x^+ _0}{f_{2^{'}}(x)}$, then $x_0$ is a point of discontinuity;
\item if $\lim_{x\to x^{-} _0 }{f_{2^{'}}(x)}=\lim_{x\to x^{+} _0 }{f_{2^{'}}(x)}$, then $f_{2^{'}}$ is continuous at $x_0$.
\end{itemize}
\end{proof}
\begin{definition}
A set of the form
$$
\Delta^{\left(2, (\theta_{k_n, i_n})\right)} _{c_1c_2...c_{k_1+k_2+...+k_n}}\equiv\left\{z: z=\Delta^{\left(2, (\theta_{k_n, i_n})\right)} _{c_1c_2...c_{k_1+k_2+...+k_n}\beta_{k_1+k_2+...+k_n+1}...\beta_{k_1+k_2+...+k_{n}+l}...}\right\},
$$
where  $c_1, \dots , c_{k_1+k_2+...+k_n}$ is a fixed binary tuple, $l=1,2,\dots$, and $\beta_{k_1+...+k_{n}+l}\in\{0,1\}$ such that 
$$
\left(\beta_{k_1+k_2+...+k_{n+l-1}+1},\dots , \beta_{k_1+k_2+...+k_{n+l}}\right)=\theta_{k_{n+l},i_{n+l}}\left(\alpha_{k_1+k_2+...+k_{n+l-1}+1},\dots , \alpha_{k_1+k_2+...+k_{n+l}}\right),
$$
is called \emph{ a pseudo-binary cylinder of rank $k_1+k_2+\dots +k_{n}$ with base ${c_1c_2...c_{k_1+k_2+...+k_n}}$}.
\end{definition}

One can note that a cylinder $\Delta^2 _{b_1b_2...b_n}$ is a closed interval,  the Lebesgue measure $\lambda\left(\Delta^2 _{b_1b_2...b_n}\right)=\left|\Delta^2 _{b_1b_2...b_n}\right|$ of $\Delta^2 _{b_1b_2...b_n}$ and the diameter $d \left(\Delta^2 _{b_1b_2...b_n}\right)=\sup\Delta^2 _{b_1b_2...b_n}-\inf \Delta^2 _{b_1b_2...b_n}$  are equal.

\begin{remark}
\label{remark-cylinders} The following considerations are useful for proving the next lemma on properties of pseudo-binary cylinders.

Let us have a binary cylinder $\Delta^2 _{b_1b_2...b_{k_1+...+k_n}}$ of rank $k_1+\dots+k_n$ with base ${b_1b_2...b_{k_1+...+k_n}}$. Here $b_1,b_2,\dots , b_{k_1+...+k_n}$ is a fixed binary tuple. That is, 
$$
\Delta^2 _{b_1b_2...b_{k_1+...+k_n}}\equiv\left\{x: x=\Delta^2 _{b_1b_2...b_{k_1+...+k_n}\alpha_{k_1+...+k_n+1}\alpha_{k_1+...+k_n+2}...}\right\},
$$
where $\alpha_j\in\{0,1\}$ for all integers $ j>k_1+k_2+\dots +k_n+1$.
Whence
$$
\Delta^{\left(2, (\theta_{k_n, i_n})\right)} _{c_1c_2...c_{k_1+k_2+...+k_n}}=f_{2^{'}}\left(\Delta^2 _{b_1b_2...b_{k_1+...+k_n}}\right)
$$
$$
=\Delta^2 _{\theta_{k_1,i_1}(b_1,...,b_{k_1})\theta_{k_2,i_2}(b_{k_1+1},...,b_{k_1+k_2})...\theta_{k_n, i_n}(b_{k_1+...+k_{n-1}+1},...,b_{k_1+...+k_n})},
$$
where $(c_{k_1+\dots+k_{j-1}+1},\dots ,c_{k_1+...+k_j})=\theta_{k_j, i_j}(b_{k_1+...+k_{j-1}+1},\dots ,b_{k_1+...+k_j})$ for all $j=\overline{1,n}$.

Since for any $\Delta^2 _{b_1b_2...b_n}$  the condition 
$$
\left |\Delta^2 _{b_1b_2...b_n}\right|=\frac{1}{2^n}
$$
holds, we obtain
$$
d\left(\Delta^{\left(2, (\theta_{k_n, i_n})\right)} _{c_1c_2...c_{k_1+k_2+...+k_n}}\right)=\left|\Delta^{\left(2, (\theta_{k_n, i_n})\right)} _{c_1c_2...c_{k_1+k_2+...+k_n}}\right|
$$
$$
=\sup\Delta^{\left(2, (\theta_{k_n, i_n})\right)} _{c_1c_2...c_{k_1+k_2+...+k_n}}-\inf\Delta^{\left(2, (\theta_{k_n, i_n})\right)} _{c_1c_2...c_{k_1+k_2+...+k_n}}
$$
$$
=\Delta^{\left(2, (\theta_{k_n, i_n})\right)} _{c_1c_2...c_{k_1+k_2+...+k_n}(1)}-\Delta^{\left(2, (\theta_{k_n, i_n})\right)} _{c_1c_2...c_{k_1+k_2+...+k_n}(0)}=\frac{1}{2^{k_1+k_2+\dots+k_n}}.
$$

So,
$$
\left|\Delta^2 _{b_1b_2...b_{k_1+...+k_n}}\right|=\left|\Delta^{\left(2, (\theta_{k_n, i_n})\right)} _{c_1c_2...c_{k_1+k_2+...+k_n}}\right|.
$$
Also, note that for any $x\in \Delta^2 _{b_1b_2...b_{k_1+...+k_n}}$ 
$$
f_{2^{'}}(x)\in f_{2^{'}}\left(\Delta^2 _{b_1b_2...b_{k_1+...+k_n}}\right)=\Delta^{\left(2, (\theta_{k_n, i_n})\right)} _{c_1c_2...c_{k_1+k_2+...+k_n}},
$$
and
$$
\inf \Delta^{\left(2, (\theta_{k_n, i_n})\right)} _{c_1c_2...c_{k_1+k_2+...+k_n}}\le f_{2^{'}}(x)\le \sup\Delta^{\left(2, (\theta_{k_n, i_n})\right)} _{c_1c_2...c_{k_1+k_2+...+k_n}}.
$$
In addition, if $x_0$ is the endpoint of $\Delta^2 _{b_1b_2...b_{k_1+...+k_n}}$, then
$$
\left|\lim_{x\to x^+ _0}{f_{2^{'}}(x)}-\lim_{x\to x^- _0}{f_{2^{'}}(x)}\right|<\frac{1}{2^{k_1+k_2+\dots+k_{n-1}}}.
$$
This condition follows from the definition of $f_{2^{'}}$, Lemma~\ref{lm: lemma-f2-continuous}, and properties of cylinders.
\end{remark}
\begin{lemma}
Cylinders $\Delta^{\left(2, (\theta_{k_n, i_n})\right)} _{c_1c_2...c_{k_1+k_2+...+k_n}}$ have the following properties:
\begin{enumerate}
\item A cylinder $\Delta^{\left(2, (\theta_{k_n, i_n})\right)} _{c_1c_2...c_{k_1+k_2+...+k_n}}$ is a closed interval and
$$
\Delta^{\left(2, (\theta_{k_n, i_n})\right)} _{c_1c_2...c_{k_1+k_2+...+k_n}}=\left[\Delta^{\left(2, (\theta_{k_n, i_n})\right)} _{c_1c_2...c_{k_1+k_2+...+k_n}(0)}, \Delta^{\left(2, (\theta_{k_n, i_n})\right)} _{c_1c_2...c_{k_1+k_2+...+k_n}(1)}\right].
$$
\item
$$
\left|\Delta^{\left(2, (\theta_{k_n, i_n})\right)} _{c_1c_2...c_{k_1+k_2+...+k_n}}\right|=\frac{1}{2^{k_1+k_2+\dots+k_n}}.
$$
\item
$$
\Delta^{\left(2, (\theta_{k_n, i_n})\right)} _{c_1c_2...c_{k_1+k_2+...+k_n+k_{n+1}}}\subset\Delta^{\left(2, (\theta_{k_n, i_n})\right)} _{c_1c_2...c_{k_1+k_2+...+k_n}}.
$$
\item For any number $x\in [0,1]$, the following is true:
$$
x=\bigcap^{\infty} _{n=1}{\Delta^{\left(2, (\theta_{k_n, i_n})\right)} _{c_1c_2...c_{k_1+k_2+...+k_n}}}.
$$
\item Cylinders $\Delta^{\left(2, (\theta_{k_n, i_n})\right)} _{c_1c_2...c_{k_1+k_2+...+k_n}}$ are situated by the rule:
$$
f_{2^{'}}:\left\{
\begin{aligned}
(\underbrace{0, 0,\ldots, 0,0}_{k_n})&\to \theta_{k_n,i_n}(\underbrace{0, 0,\ldots, 0,0}_{k_n})\\
(\underbrace{0, 0,\ldots, 0,1}_{k_n})&\to \theta_{k_n,i_n}(\underbrace{0, 0,\ldots, 0,1}_{k_n})\\
 \dots \dots \dots \\
(\underbrace{1, 1,\ldots, 1,1}_{k_n})&\to \theta_{k_n,i_n}(\underbrace{1,1,\ldots, 1,1}_{k_n})\\
\end{aligned}
\right.
$$
\end{enumerate}
\end{lemma}
\begin{proof}\begin{enumerate}
\item Suppose $x\in \Delta^{\left(2, (\theta_{k_n, i_n})\right)} _{c_1c_2...c_{k_1+k_2+...+k_n}}$. Then
$$
x^{'}=\Delta^{\left(2, (\theta_{k_n, i_n})\right)} _{c_1c_2...c_{k_1+k_2+...+k_n}(0)}\le x\le \Delta^{\left(2, (\theta_{k_n, i_n})\right)} _{c_1c_2...c_{k_1+k_2+...+k_n}(1)}=x^{''}.
$$
Hence $x\in[x^{'},x^{''}]$ and $\Delta^{\left(2, (\theta_{k_n, i_n})\right)} _{c_1c_2...c_{k_1+k_2+...+k_n}}\subseteq [x^{'},x^{''}]$. Since
$$
\Delta^{\left(2, (\theta_{k_n, i_n})\right)} _{c_1c_2...c_{k_1+k_2+...+k_n}(0)}=\sum^{k_1+\dots+k_n} _{j=1}{\frac{c_j}{2^j}}+\frac{1}{2^{k_1+k_2+\dots+k_n}}\inf\sum^{\infty} _{t=k_1+...+k_n+1}{\frac{\beta_t}{2^t}}
$$
and
$$
\Delta^{\left(2, (\theta_{k_n, i_n})\right)} _{c_1c_2...c_{k_1+k_2+...+k_n}(1)}=\sum^{k_1+\dots+k_n} _{j=1}{\frac{c_j}{2^j}}+\frac{1}{2^{k_1+k_2+\dots+k_n}}\sup\sum^{\infty} _{t=k_1+...+k_n+1}{\frac{\beta_t}{2^t}},
$$
we obtain that $x, x^{'}, x^{''}\in \Delta^{\left(2, (\theta_{k_n, i_n})\right)} _{c_1c_2...c_{k_1+k_2+...+k_n}}$. So a cylinder $\Delta^{\left(2, (\theta_{k_n, i_n})\right)} _{c_1c_2...c_{k_1+k_2+...+k_n}}$  is a closed interval.

 \item This property follows from the first statement of this lemma and from Remark~\ref{remark-cylinders}.

\item Let $m$ be an arbitrary fixed positive integer. The following conditions
$$
\left\{
\begin{aligned}
\inf\Delta^{\left(2, (\theta_{k_n, i_n})\right)} _{c_1c_2...c_{k_1+k_2+...+k_m+k}}&\ge\inf\Delta^{\left(2, (\theta_{k_n, i_n})\right)} _{c_1c_2...c_{k_1+k_2+...+k_m}}\\
\sup\Delta^{\left(2, (\theta_{k_n, i_n})\right)} _{c_1c_2...c_{k_1+k_2+...+k_m+k}}&\le\sup\Delta^{\left(2, (\theta_{k_n, i_n})\right)} _{c_1c_2...c_{k_1+k_2+...+k_m}}\\
\end{aligned}
\right.
$$
are satisfied, indeed we have
$$
\inf\Delta^{\left(2, (\theta_{k_n, i_n})\right)} _{c_1c_2...c_{k_1+k_2+...+k_m+k}}-\inf\Delta^{\left(2, (\theta_{k_n, i_n})\right)} _{c_1c_2...c_{k_1+k_2+...+k_m}}=\Delta^{\left(2, (\theta_{k_n, i_n})\right)} _{c_1c_2...c_{k_1+k_2+...+k_m+k}(0)}-\Delta^{\left(2, (\theta_{k_n, i_n})\right)} _{c_1c_2...c_{k_1+k_2+...+k_m}(0)}\ge 0
$$
and
$$
\sup\Delta^{\left(2, (\theta_{k_n, i_n})\right)} _{c_1c_2...c_{k_1+k_2+...+k_m}}-\sup\Delta^{\left(2, (\theta_{k_n, i_n})\right)} _{c_1c_2...c_{k_1+k_2+...+k_m+k}}=\Delta^{\left(2, (\theta_{k_n, i_n})\right)} _{c_1c_2...c_{k_1+k_2+...+k_m}(1)}-\Delta^{\left(2, (\theta_{k_n, i_n})\right)} _{c_1c_2...c_{k_1+k_2+...+k_m+k}(1)}\ge 0.
$$

\item As a consequence of the previous properties we have the following chain of inclusions
$$
\Delta^{\left(2, (\theta_{k_n, i_n})\right)} _{c_1c_2...c_{k_1}}\subset \Delta^{\left(2, (\theta_{k_n, i_n})\right)} _{c_1c_2...c_{k_1+k_2}}\subset \ldots \subset \Delta^{\left(2, (\theta_{k_n, i_n})\right)} _{c_1c_2...c_{k_1+k_2+...+k_n}}\subset \ldots ,
$$
therefore from Cantor's intersection theorem we obtain
$$
\bigcap^{\infty} _{n=1}{\Delta^{\left(2, (\theta_{k_n, i_n})\right)} _{c_1c_2...c_{k_1+k_2+...+k_n}}}=x=\Delta^{\left(2, (\theta_{k_n, i_n})\right)} _{c_1c_2...c_n...}\in [0,1].
$$

\item This property follows from the definition of $f_{2^{'}}$.
\end{enumerate}
\end{proof}
\begin{remark}
Let us remark that properties of the pseudo-s-adic numeral system can be described in terms of the $s$-adic numeral system or in terms of such modified numeral systems. Formulations of certain properties depend on such considerations. In particular, number representations  having two different representations, as well as  situating cylinders of rank $n$ are these properties. We can see these properties by the definition of the relationship between the $s$-adic and  pseudo-$s$-adic representation.
 Now we give an example for the case of the nega-binary representation. In this article, properties of the modified $s$-adic representations are more  investigated in terms of the $s$-adic (the binary or  ternary) representations. 
\end{remark}

\begin{example}
The nega-binary representation of $x \in \left[-\frac{2}{3},\frac{1}{3}\right]$  is defined as 
$$
x=\Delta^{-2} _{\beta_1\beta_2...\beta_n...}\equiv -\frac{\beta_1}{2}+\frac{\beta_2}{2^2}-\frac{\beta_3}{2^3}+\dots+\frac{\beta_n}{(-2)^n}+\dots, ~~~~~\beta_n\in\{0,1\}.
$$
For this representation the following equality holds
$$
\Delta^{-2} _{\beta_1\beta_2...\beta_n...}+\Delta^2 _{(10)}=\Delta^2 _{[1-\beta_1]\beta_2...[1-\beta_{2n-1}]\beta_{2n}...}\equiv\Delta^2 _{\alpha_1\alpha_2...\alpha_n...} \in[0,1].
$$
In terms of the binary representation, 
$$
\Delta^2 _{\alpha_1\alpha_2...\alpha_{n-1}\alpha_n(0)}=\Delta^2 _{\alpha_1\alpha_2...\alpha_{n-1}[\alpha_n-1](1)},
$$
where $\alpha_n\ne 0$ and $\alpha_n=1$ (since $\alpha_n\in\{0,1\}$). But in terms of the nega-binary representation
$$
\beta_n=\begin{cases}
1-\alpha_n&\text{if $n$ is odd}\\
\alpha_n,&\text{if $n$ is even}.
\end{cases}
$$
Whence
$$
\Delta^{-2} _{\beta_1\beta_2...\beta_{n-1}\beta_n(10)}=\Delta^{-2} _{\beta_1\beta_2...\beta_{n-1}[1-\beta_n](01)}, ~~~\beta_n=1.
$$

Similarly, cylinders $\Delta^2 _{b_1b_2...b_n}$ are left-to-right situated for any $n\in\mathbb N.$ But cylinders $\Delta^{-2} _{[1-b_1]b_2...[1-b_{2n-1}]b_{2n}}$ are  left-to-right situated and cylinders $\Delta^{-2} _{[1-b_1]b_2...[1-b_{2n-1}]}$ are right-to-left situated. 

In addition, the map $f_{2^{'}}$ can be defined by the following table 
\begin{center}
\begin{tabular}{|c|c|c|c|c|c|}
\hline
$\alpha_{2m+1}\alpha_{2(m+1)}$ & $ 00 $& $01$ & $10$ & $11$\\
\hline
$\beta_{2m+1}\beta_{2(m+1)}$ & $10$ & $11$ & $00$ & $01$\\
\hline
\end{tabular}
\end{center}
\end{example}

Let us return to the consideration of  the map \eqref{eq: 11}.

\begin{theorem}
A map $f_{2^{'}}$ has the following properties (characteristics):
\begin{enumerate}
\item Suppose $[a,b]$ is a closed interval. Then the set $f_{2^{'}}\left([a,b]\right)$ is a closed interval or a union of closed intervals with a finite set of isolated points, or a union of closed intervals.
\item A map $f_{2^{'}}$ preserves the Lebesgue measure of intervals.
\item A map $f_{2^{'}}$, except for the cases when $f_{2^{'}}(x)=x$ or $ f_{2^{'}}(x)=1-x$,  does not preserve distance between points on $[0,1]$.
\end{enumerate}
\end{theorem}
\begin{proof} Let us prove the first and the second statements. 

Let $[a,b]$ be a certain closed interval. Then from the last lemma it follows that $f_{2^{'}}\left([a,b]\right)$ is a closed interval whenever $[a,b]$ is a certain binary cylinder $ \Delta^2 _{b_1b_2...b_{k_1+...+k_n}}$.

Let us consider the case when  $[a,b]$ is not a binary cylinder. Then there exist  coverings $[a,b]$ by binary cylinders $ \Delta^2 _{b_1b_2...b_{k_1+...+k_l}}$, $\Delta^2 _{b_1b_2...b_{k_1+...+k_m}}$ such that
$$
[a,b]\subseteq\bigcup_{l}{\Delta^2 _{b_1b_2...b_{k_1+...+k_l}}},
$$
$$
[a,b]\supseteq\bigcup_{m}{\Delta^2 _{b_1b_2...b_{k_1+...+k_m}}},
$$
and
$$
\lim_{l\to\infty}{\left(\bigcup_{l}{\Delta^2 _{b_1b_2...b_{k_1+...+k_l}}}\right)}=\lim_{m\to\infty}{\left(\bigcup_{m}{\Delta^2 _{b_1b_2...b_{k_1+...+k_m}}}\right)}=\lambda\left([a,b]\right).
$$

Since $f_{2^{'}}\left(\Delta^2 _{b_1b_2...b_{k_1+...+k_n}}\right)$ is a closed interval, $\lambda\left(f_{2^{'}}\left(\Delta^2 _{b_1b_2...b_{k_1+...+k_n}}\right)\right)=\lambda\left(\Delta^2 _{b_1b_2...b_{k_1+...+k_n}}\right)$, and an arbitrary binary rational point can be a jump discontinuity of $f_{2^{'}}$, we obtain that the first and the second  properties are true.

To prove the third statement, let us consider mapping two adjacent cylinders of rank $k_1+k_2+\dots+k_n$ under the action of $f_{2^{'}}$. If such two adjacent cylinders map to non-adjacent cylinders under $f_{2^{'}}$, then one can assume that a distance between points of these cylinders is not preserved  under the map $f_{2^{'}}$. 

It is easy to see that all pairwise adjacent cylinders are pairwise adjacent under $f_{2^{'}}$ whenever the condition $(\theta_{k_n,i_n})=\theta_{k_n,0}$ holds for all $n\in\mathbb N$ or $(\theta_{k_n,i_n})=\theta_{k_n,(2^{k_n})!-1}$ holds for all $n\in\mathbb N$. That is, when $f_{2^{'}}(x)=x$ or $f_{2^{'}}(x)=1-x$. In the other case, there exist at least one pair of adjacent cylinders such that their images under $f_{2^{'}}$ are not adjacent cylinders. That is, from the existence of jump discontinuities of $ f_{2^{'}}$ it follows that our statement is true.
\end{proof}

The next lemma follows from the definition of $ f_{2^{'}}$, the existence of jump discontinuities of $ f_{2^{'}}$, and from the placement of adjacent cylinders of the same rank. 
\begin{lemma}
A map $ f_{2^{'}}$:
\begin{itemize}
\item is strictly monotonic whenever $ f_{2^{'}}=x$ or $ f_{2^{'}}=1-x$; that is, when for all positive integers $n$  the following  conditions hold: $\theta_{k_n,i_n}=\theta_{k_n,0}$ or $\theta_{k_n,i_n}=\theta_{k_n,(2^{k_n})!-1}$.
\item has intervals of the monotonicity whenever almost all elements of $(\theta_{k_n,i_n})$ are equal only to $\theta_{k_n,0}$ or $\theta_{k_n,(2^{k_n})!-1}$.
\item is not monotonic in another case.
\end{itemize}
\end{lemma}


Let us consider the random variable $\eta$ defined as
$$
\eta= \Delta^{\left(2, (\theta_{k_n, i_n})\right)} _{\xi_1\xi_2...\xi_n...},
$$
where digits $\xi_n$ are random and taking the values $0,1$ with
probabilities $p_0,p_1$. Here $p_0+p_1=1$. That is, $\xi_n$ are independent, and $P\{\xi_n=i_n\}=p_{i_n}$, where $i_n\in\{0,1\}$. Then the distribution function $F_{\eta}$ of the random variable $\eta$ is
$$
{F}_{\eta}(x)
=a_{\beta_1(x)}+\sum\limits^{\infty} _{n=2}
{\left({a}_{\beta_n(x)} \prod\limits^{n-1} _{j=1} {{p}_{\beta_j(x)}}\right)},
$$
where $x=\Delta^{\left(2, (\theta_{k_n, i_n})\right)} _{\beta_1\beta_2...\beta_n...}\in [0,1]$. 

Note that one can consider the case when $P\{\xi_n=i_n\}=p_{i_n,n }$, where for any $n\in\mathbb N$ $p_{0,n}+p_{1,n}=1$. Then 

$$
{F}_{\eta}(x)
=a_{\beta_1(x), 1}+\sum\limits^{\infty} _{n=2}
{\left({a}_{\beta_n(x),n} \prod\limits^{n-1} _{j=1} {{p}_{\beta_j(x),j}}\right)},
$$
where $0 \le x\ge 1$.

Note also that a function of the form
$$
f_D(x)=a_{\beta_1(x)}+\sum\limits^{\infty} _{n=2}
{\left({a}_{\beta_n(x)} \prod\limits^{n-1} _{j=1} {{p}_{\beta_j(x)}}\right)},
$$
where $x=\Delta^2 _{\alpha_1\alpha_2...\alpha_n...}$ and 
$$
 (\beta_{k_1+...+k_{n-1}+1}, \dots ,\beta_{k_1+\dots+k_n})=\theta_{k_n,i_n} (\alpha_{k_1+...+k_{n-1}+1}, \dots ,\alpha_{k_1+...+k_n}),
$$
 $n=1,2,\dots ,$  and $k_0=0$, is a generalization of the Salem function~(\cite{Salem}; the Salem function is one of the simplest examples of singular functions. In addition,
$$
f_D=F_{\eta}\circ f_{2^{'}}.
$$

Now let us consider  integral properties of $f_{2^{'}}$. 

\begin{theorem}
The Lebesgue integral of the function $f_{2^{'}}$   is equal to $\frac{1}{2}$.
\end{theorem}
\begin{proof}
Since $0\le f_{2^{'}}(x)\le1$, we choose 
$$
E_n=\{x: y_{n-1}\le f_{2^{'}}(x)<y_n\}=\Delta^2 _{b_1b_2...b_{k_1+k_2+...+k_n}},
$$
where
$$
|E_n|=\frac{1}{2^{k_1+k_2+\dots+k_n}}
$$
and
$$
T=\left\lbrace
0, \Delta^{\left(2, (\theta_{k_n, i_n})\right)} _{c_1c_2...c_{k_1}(0)}, \Delta^{\left(2, (\theta_{k_n, i_n})\right)} _{c_1c_2...c_{k_1+k_2}(0)}, \dots \right\rbrace.
$$

Clearly, 
$ y_{n-1}\in[y_{n-1},y_n)$ and the conditions $\lambda(E_n)\to 0$ and $n\to\infty$ are equivalent. Then
$$
I=\lim_{n\to\infty}{\left(\sum_{\text{}~ c_i\in\{0,1\}}{\frac{\Delta^{\left(2, (\theta_{k_n, i_n})\right)} _{c_1c_2...c_{k_1+k_2+...+k_n}(0)}}{2^{k_1+...+k_n}}}\right)}=\lim_{n\to\infty}{\frac{2^{k_1+...+k_n}-1}{2\cdot 2^{k_1+...+k_n}}}=\frac 1 2,
$$
$i=1,2, \dots , k_1+k_2+\dots +k_n$.
\end{proof}

\section{Further investigations and open problems}
\label{section 4}

This section is devoted to introducing open problems  and further investigations which the author of this paper is focusing on for future developments in his own research.

This article is an initial article in the series of papers of the author of the present article devoted to the investigation of the operator approach for modeling, studying, and applying different new numeral systems (positive and alternating expansions of real numbers with a finite,  infinite,  or variable alphabet, etc.).

One can note that the approach  described in the present article for modeling numeral systems can be used for the case  of the $s$-adic numeral system with a fractional base $s>1$ (such representation of real numbers was introduced in \cite{Renyi1957}) or for the case  of the nega-$s$-adic numeral system with a fractional base $(-s)<-1$ (this representation of real numbers is introduced in \cite{IS2009}). Also, the approach can be applied to numeral systems with a negative (integer or fractional) base and/or with a variable alphabet (we have a variable alphabet for a certain   representation $\Delta_{\gamma_1\gamma_2...\gamma_n...}$, where $\gamma_n\in A_n$,  when there exists a finite or infinite number of pairs $l\ne m$ such that $|A_l|\ne|A_m|$). It is easy to see that a basis of a numeral system with a variable alphabet is a sequence or a matrix. 

In the next articles of the author of the present article, the approach introduced in Section~\ref{section 2a} will be used and investigated for the cases of all these representations of real numbers. 

Note that examples of representations with a variable alphabet are representations of real numbers by positive (\cite{Cantor1}) or alternating (\cite{Serbenyuk2017, Symon2017}) Cantor series, the nega-$\tilde Q$- (\cite{Serbenyuk2016}) or $\tilde Q^{'} _{\mathbb N_B}$-representation~(\cite{S.Serbenyuk 2018}), etc.

Certain problems are also  related to Lemma~\ref{lm: lemma 2-problems}. Let us consider these
problems introducing them in the present article.  

In \cite{Renyi1957}, the $f$-expansion of real numbers was considered:
$$
x=f(\varepsilon_1+f(\varepsilon_2+f(\varepsilon_3+\dots)))\equiv\Delta^f _{\varepsilon_1\varepsilon_2...\varepsilon_n...},
$$
where $f$ is a fixed function having certain properties. Notice that it is proved that such function is increasing. A number of researches devoted to different cases of the $f$-expansion (for example, see \cite{{Renyi1957}, {Schweiger(184)2018}, {Schweiger2016}}, etc.). 

For example, if $f(x)=\frac{x}{\beta}$, where $\beta>1$ is a fixed real number, then 
$$
x=\sum^{\infty} _{n=1}{\frac{\varepsilon_n}{\beta^n}}\equiv\Delta^{\beta} _{\varepsilon_1\varepsilon_2...\varepsilon_n...}. 
$$
The last-mentioned expansion (see \cite{Renyi1957}) of real numbers is called \emph{the $\beta$-expansion}. 

If $f(x)=-\frac{x}{\beta}$, where $(-\beta)<-1$, then
$$
x=\sum^{\infty} _{n=1}{\frac{\varepsilon_n}{(-\beta)^n}}\equiv\Delta^{-\beta} _{\varepsilon_1\varepsilon_2...\varepsilon_n...}. 
$$
The last-mentioned expansion of real numbers is called \emph{the $(-\beta)$-expansion} and  was introduced in \cite{{IS2009}}.

\begin{problem}
Let us note that we can obtain our case of the pseudo-ternary representation by the following way.  Suppose that
$f(\alpha_n)=\frac{\theta(\alpha_n)}{s}$, where $s\ge 2$ is a natural
number and $\theta=\theta_{1,1}$. Then the obtained $f$-expansion is following:
$$
\Delta^{3} _{\theta(\alpha_1)\theta(\alpha_2)...\theta(\alpha_n)...}=\Delta^{3^{'}} _{\beta_1\beta_2...\beta_n...}.
$$
\begin{conjecture}
One can model   the $f$-expansion such that all points except points from no more countable set (i.e., we deal with a set of the full Lebesgue measure) have the unique representation when $f$ is  is determined at any point and continuous almostwhere (having jump discontinuities) on the domain. Also, $f$ can be non-monotone or non-differentiable.
\end{conjecture}

In addition, a general technique representing the preudo-$s$-adic representation by the $f$-expansion is unknown. 

Let us note that we can write the pseudo-ternary representation in the following
way:
$$
\Delta^{3^{'}} _{\beta_1\beta_2...\beta_n...}=\Delta^{f(x)} _{\beta_1\beta_2...\beta_n...}=f\left(\Delta^3 _{\alpha_1\alpha_2...\alpha_n...}\right),
$$
where $x=\Delta^3 _{\alpha_1\alpha_2...\alpha_n...}$, $f$ is defined by equality \eqref{eq: function f},  and
$$
\Delta^{f(x)} _{\beta_1\beta_2...\beta_n...}=f\left(\Delta^{3} _{\theta(\alpha_1)\theta(\alpha_2)...\theta(\alpha_n)...}\right).
$$
\end{problem}

\begin{problem}
Let $(f_n)$ be a fixed sequence of  certain functions. One can define the $(f_n)$-expansion of real numbers by the following way:
$$
x=f_1(\varepsilon_1+f_2(\varepsilon_2+f_3(\varepsilon_3+\dots)))\equiv\Delta^{(f_n)} _{\varepsilon_1\varepsilon_2...\varepsilon_n...}.
$$
This representation is new and is introduced in this paper. However there are several examples of known representations obtained from the $(f_n)$-expansion.  
\end{problem}
\begin{example}
If $f_n(x)=\frac{x}{q_n}$, where $(q_n)$ is a fixed sequence of positive integers and $q_n>1$ for any $n\in\mathbb N$, then we obtain the representation of real numbers by positive Cantor series \cite{Cantor1}.
\end{example}
\begin{problem}
Suppose that $f_n(x)=\frac{x}{q_n}$, where $(q_n)$ is a fixed sequence of numbers for which one of the following conditions holds: $q_n>1$ or $q_n\ge 2$.  Then we obtain the encoding of real numbers by positive Cantor series (with a fractional base). Such positive and alternating expansions of real numbers are new and introduced in this article.  Note that we can define such representations by the condition $\varepsilon_n\in\{0,1,\dots , [q_n]\}$, $\varepsilon_n\in\{0,1,\dots , [q_n-1]\}$,  or $\varepsilon_n\in\{0,1,\dots , [q_n]-1\}$. Here $[a]$ is the integer part of $a$. The cases of these conditions are interesting for future investigations. 

Note that it is easy to see that one can model the known  $Q_s$-, $Q^{*} _s$-, $\tilde Q$- and other  representations by analogy with this case.
\end{problem}

\begin{example}[{\bfseries Quasi-nega-representations}] 
Let $\beta>1$ be a fixed real number and $(q_n)$ be a fixed sequence of positive integers such that $q_n\ge 2$.

Let us have 
$$
f(x)=\frac{x}{\beta(-1)^{\rho_n+\rho_{n-1}}}
$$ 
 and a fixed sequence $(f_n)$ of functions 
$$
f_n(x)=\frac{x}{q_n(-1)^{\rho_{n-1}+\rho_n}},
$$ 
 where
$$
\rho_n=\begin{cases}
1&\text{if $n\in \mathbb N_B$}\\
2&\text{if  $n\notin \mathbb N_B$,}
\end{cases}
$$
$\mathbb N_B$ is a fixed subset of positive integers, and $\rho_0=0$. Then we get the following two quasi-nega-expansions with a fractional base in the general case:
$$
x=\sum^{\infty} _{n=1}{\frac{\varepsilon_n(-1)^{\rho_n}}{\beta^{n}}},
$$
$$
x=\sum^{\infty} _{n=1}{\frac{(-1)^{\rho_n}\varepsilon_n}{q_1q_2\cdots q_n}}.
$$

Note that we can model the quasi-nega-$\tilde Q$-expansion by analogy. 
\end{example}

\begin{conjecture}
Properties of $f$ and of $(f_n)$  that are sufficient for modeling the $f$- and $(f_n)$-representations coincide.
\end{conjecture}

Considering  Section~\ref{section 3a}, one can note the  the following conjecture.
\begin{conjecture}
The function $F_{\eta}$ is singular.
\end{conjecture}

These mentioned  problems and new representations of real numbers, as well as  $F_{\eta}, f_D$, and $f_{2^{'}}$ described in Section~\ref{section 3a} 
and  their generalizations will be considered and investigated in the next papers of the author of the present article.

\end{document}